\documentclass[12pt]{amsart}%
\usepackage{bm,graphics,epsfig,color,float,latexsym,array,amsthm,mathrsfs,amssymb,amsfonts,graphicx,amsmath,psfrag,enumerate}%
\usepackage{epic,fancybox}
\usepackage[usenames,dvipsnames]{pstricks}
 \usepackage{pst-grad} 
 \usepackage{pst-plot} 
\usepackage{tikz}
\usetikzlibrary{arrows,automata,positioning}
\usepackage[left=3.0cm,right=3.0cm,top=3.7cm,bottom=3.7cm]{geometry}

\numberwithin{equation}{section}
\newtheorem{thm}{Theorem}[section]
\newtheorem{prop}[thm]{Proposition}
\newtheorem{lem}[thm]{Lemma}
\newtheorem{cor}[thm]{Corollary}

\theoremstyle{definition}
\newtheorem{fact}[thm]{Fact}
\newtheorem{defn}[thm]{Definition}
\newtheorem*{defn*}{Definition}
\newtheorem{ex}[thm]{Example}

\newtheorem{que}[thm]{Question}

\theoremstyle{remark}
\newtheorem{rem}[thm]{Remark}
\newtheorem*{ack*}{Acknowledgment}

\newcommand{\Z}{\mathbb{Z}}
\newcommand{\ra}{\rightarrow}
\newcommand{\wt}{\tilde}
\newcommand{\B}{\mathcal{B}}

\newcommand{\C}{\mathscr{C}}
\newcommand{\A}{\mathcal{A}}
\newcommand{\piinv}{\pi^{-1}}

\newcommand{\tr}{\mathsf{T}}
\newcommand{\setN}{\mathbb{N}}
\newcommand{\setZ}{\mathbb{Z}}

\begin{document}

\title[Transition classes]{Structure of transition classes for factor codes on shifts of finite type}

\author[M. Allahbakhshi]{Mahsa Allahbakhshi}
\address{Centro de Modelamiento Matem\'atico \\
    Universidad de Chile \\
    Av. Blanco Encalada 2120, Piso 7 \\
    Santiago de Chile \\
    Chile}
\email{mallahbakhshi@dim.uchile.cl}

\author[S. Hong]{Soonjo Hong}
\address{Centro de Modelamiento Matem\'atico \\
    Universidad de Chile \\
    Av. Blanco Encalada 2120, Piso 7 \\
    Santiago de Chile \\
    Chile}
\email{hsoonjo@dim.uchile.cl}

\author[U. Jung]{Uijin Jung}
\address{Department of Mathematics \\
	Ajou University \\
    Suwon 443-749 \\
	South Korea}
\email{uijin@ajou.ac.kr}

\date{}
\begin{abstract}
Given a factor code $\pi$ from a shift of finite type $X$ onto a sofic shift $Y$, the class degree of $\pi$ is defined to be the minimal number of transition classes over the points of $Y$. In this paper we investigate the structure of transition classes and present several dynamical properties analogous to the properties of fibers of finite-to-one factor codes. As a corollary, we show that for an irreducible factor triple there cannot be a transition between two distinct transition classes over a right transitive point, answering a question raised by Quas.
\end{abstract}
\maketitle

\section{Introduction}

Given a finite-to-one factor code $\pi$ from a shift of finite type $X$ onto an irreducible sofic shift $Y$, the \emph{degree} of $\pi$ is defined to be the minimal number of preimages of the points in $Y$. The notion of degree was first introduced in \cite{Hed69} for endomorphisms of full shifts, and then was extended to those of irreducible shifts of finite type and sofic shifts \cite{CovP74}. It is widely studied and is useful in the study of finite-to-one factor codes \cite{Ash90, Boy86, KitMT91, Tro90}.
If $d$ is the degree of a one-block factor code $\pi$,
there are well-known fundamental properties of fibers of points in $Y$ as follows:
\begin{enumerate}
    \item Every doubly transitive point in $Y$ has exactly $d$ preimages.
    \item $\pi(x)$ is doubly transitive if and only if so is $x$.
    \item Any two distinct preimages of a doubly transitive point in $Y$ are \emph{mutually separated}; i.e., they do not share a common symbol at the same time.
\end{enumerate}

In this work we show that a natural generalization of the degree for general factor codes is the \emph{class degree} introduced in~\cite{AllQ13}. The principal motivation of defining the class degree was finding a conjugacy-invariant upper bound on the number of ergodic measures of relative maximal entropy. Measures of relative maximal entropy appeared in many different topics in symbolic dynamics due to their connections with, for example, functions of Markov chains~\cite{Bla57,BurR58} or their application in computing the Hausdorff dimension of certain sets~\cite{GatP97}. The class degree is in fact a conjugacy-invariant upper bound on the number of such measures over a fully supported ergodic measure~\cite{AllQ13,PetQS03}.

The class degree is defined by using a certain equivalence relation on the fiber of each point in $Y$. Roughly speaking, two preimages $x$ and $\bar x$ of a point $y$ in $Y$ are equivalent if we can find a preimage $z$ of $y$ which is equal to $x$ up to an arbitrarily large given positive coordinate and right asymptotic to $\bar x$ and vice versa (see Definition~\ref{defn:equiv}). The \emph{class degree} is defined to be the minimal number of equivalence classes (called \emph{transition classes}) over the points in $Y$. It is shown in~\cite{AllQ13} that the class degree is equal to the degree when $\pi$ is finite-to-one and moreover, if $d$ is the class degree
of $\pi$ then every right transitive point in $Y$ has exactly $d$ transition classes (analogous to (1) above). This suggests that the class degree is a candidate of a generalization of the degree.

The idea of considering transitions between preimages of a point also appeared independently in \cite{Yoo10}. To find a condition which is invariant under conjugacy and weaker than the condition which appeared in \cite{ChaU03}, Yoo defined the notion of \emph{fiber-mixing} and showed that a fiber-mixing code between two mixing shifts of finite type sends every fully supported Markov measure on $X$ to a Gibbs measure on $Y$. Fiber-mixing codes were investigated in further studies, e.g., \cite{Jun13, Shi11}. In terms of our terminology, a fiber-mixing code from a shift of finite type $X$ onto a sofic $Y$ is just a code in which every point in $Y$ has only one transition class, thus it is natural to ask what kind of properties a code can have when the code is not fiber-mixing, for example, when the code has class degree one but there exist some points with more than one transition class. Moreover, since the definition of transition classes is motivated by communicating classes in Markov chains, Quas asked
whether there could be a transition between two distinct transition classes over a right transitive point.

To answer such questions one need to have a structural theory on transition classes. In fact unlike the finite-to-one case where the fibers have been well studied, previous research on infinite-to-one codes usually had concentrated on their thermodynamic formalism \cite{BoyP11,Wal86}, or on the construction of codes with nice properties \cite{Boy83, BoyT84, Tho04}. In this paper, we investigate the fibers and transition classes of such factor codes and provide several structural results. By these results it is natural to consider the class degree as a natural generalization of the degree.

In particular, we provide dynamical properties analogous to (2) and (3) above (see Theorem \ref{thm:contain_transition_pt}, Corollary \ref{cor:contain_dbly_transition_pt}, and Theorem \ref{thm:mutually_separated}):
\begin{enumerate}
    \item[(2$'$)] $y$ is right (resp. doubly) transitive if and only if each transition class over $y$ contains a right (resp. doubly) transitive point.
    \item[(3$'$)] Given two points from two distinct transition classes over a right transitive point, they are mutually separated.
\end{enumerate}

This analogy shows us that as for a finite-to-one code, fibers over almost all images for infinite-to-one factors are well-behaved, in the sense that among the fiber over a typical point of $Y$ a typical point of $X$ always exists, and that any points chosen from distinct classes over a typical point of $Y$ have orbits which neither meet nor approach asymptotically to each other. As a corollary, we also show that there cannot be any transition among distinct classes (see Corollary \ref{thm:no_transition}). Property (3$'$) above is one type of separation property between distinct transition classes. In Section 5, we present another type of separation property which is stronger than the former one; briefly, we make a partition on the set of preimages of a magic block of $\pi$ so that one can determine whether two preimages of a doubly transitive point are in the same transition class or not only by reading the symbols occurring in their coordinates over the magic block (see Theorem \ref{thm:block_partition}).
Other structural properties on transition classes will be also provided which would hopefully open new doors on further investigation of infinite-to-one factor codes.

\section{Background}
In this section we introduce some terminology and basic results on symbolic dynamics.

If $X$ is a subshift (or shift space) with the shift map $\sigma$, then denote by $\B_n(X)$ the set of all $n$-blocks occurring in the points of $X$ and let $\B(X) = \bigcup_{n = 0}^{\infty} \B_n(X)$. The alphabet of a shift space $X$ is denoted by $\A(X) = \B_1(X)$. 

A \emph{code} $\pi : X \to Y$ is a continuous $\sigma$-commuting map between shift spaces. It is called a \emph{factor code} if it is surjective. Every code can be recoded to be a one-block code; i.e., a code for which $x_0$ determines $\pi(x)_0$. Given a one-block code $\pi: X \to Y$, it naturally induces a map on $\B(X)$, which we also denote by $\pi$ for brevity. We call $\pi$ \emph{finite-to-one} if $\pi^{-1}(y)$ is a finite set for all $y \in Y$.

A triple $(X,Y,\pi)$ is called a \emph{factor triple} if $\pi:X\to Y$ is a factor code from a shift of finite type $X$ onto a (sofic) subshift $Y$. A factor triple is called irreducible when $X$ is irreducible. It is called finite-to-one if $\pi$ is finite-to-one.

A point $x$ in a shift space $X$ is called \emph{right transitive} if every block in $X$ occurs infinitely many times in $x_{[0,\infty)}$, or equivalently, if the forward orbit of $x$ is dense.  Two points $x, \bar x$ in $X$ are called \emph{right asymptotic} if $x_{[N,\infty)} =  \bar x_{[N,\infty)}$ for some $N \in \setZ$. Left asymptotic points and left transitivity are defined similarly. A point is \emph{doubly transitive} if it is both left and right transitive.

If $\pi$ is a finite-to-one factor code from a shift of finite type $X$ onto a sofic shift $Y$, there is a uniform upper bound on the number of preimages of points in $Y$ \cite{LM}. The minimal number of $\pi$-preimages of the points in $Y$ is called the \emph{degree} of the factor code $\pi$ and is denoted by $d_\pi$.

\begin{thm}\label{thm:samenumber}\cite[\S 9]{LM}
    Let $(X,Y,\pi)$ be a finite-to-one factor triple with $Y$ irreducible. Then every doubly transitive point of $Y$ has exactly $d_{\pi}$  preimages.
\end{thm}

Two points $x$ and $\bar x$ in a shift space are \emph{mutually separated} if $x_i$ and $\bar x_i$ are different for each integer $i$. It is well known that if $(X,Y,\pi)$ is a finite-to-one factor triple with $X$ one-step, $\pi$ one-block and $Y$ irreducible, then each $y \in Y$ has $d_\pi$ mutually separated preimages. In particular, if $y$ is doubly transitive, then any two distinct preimages of $y$ are mutually separated.

We say two factor triples $(X,Y,\pi)$ and $(\wt X,\wt Y,\wt\pi)$ are \emph{conjugate} if $X$ is conjugate to $\wt X$ under a conjugacy $\phi$, $Y$ is conjugate to $\wt Y$ under a conjugacy $\psi$, and $\wt\pi\circ\phi=\psi\circ\pi$. An immediate corollary of Theorem~\ref{thm:samenumber} is that $d_\pi$ is invariant under conjugacy.
%


\begin{defn}\cite{LM}
    Let $(X,Y,\pi)$ be a factor triple with $X$ one-step and $\pi$ one-block. Given a block $w\in\B(Y)$, define
        $$d(w) = \min_{1 \leq k \leq |w|} \left|  \{ a \in \A(X) : \exists u \in \pi^{-1}(w) \text{ with } u_k = a \} \right|$$

    If a block $w$ satisfies $d(w) = \min_{v \in \B(Y)} d(v)$, then it is called a \emph{magic block}. In this case, a coordinate $k$ where the minimum occurs is called a \emph{magic coordinate}. If $|w| = 1$ then $w$ is called a \emph{magic symbol}.
\end{defn}

Let $w$ be a magic block with a magic coordinate $k$. Given a point $y$ in $Y$ and $i \in \setZ$ with $y_{[i,i+|w|)}=w$, some block in $\pi^{-1}(w)$ may not be extendable to a point in $\piinv(y)$. However, due to the minimality of $d(w)$, the two sets $\{x_{i+k}\mid x\in\pi^{-1}(y)\}$ and $\{u_k\mid u\in\pi^{-1}(w)\}$ are the same.
It is well known that for a one-block finite-to-one factor code $\pi$ from a one-step shift of finite type $X$ onto an irreducible sofic shift $Y$, we have $d_\pi = d(w)$ for any magic block $w$ of $\pi$.


The class degree defined below is a quantity analogous to the degree when the factor code $\pi$ is not only limited to be finite-to-one.

\begin{defn}\label{defn:equiv}
    Let $(X,Y,\pi)$ be a factor triple and $x,\bar x \in X$. We say there is a \emph{transition} from $x$ to $\bar x$ and denote it by $x\to \bar x$ if for each integer $n$, there exists a point $z$ in $X$ so that
    \begin{enumerate}
        \item $\pi(z)=\pi(x)=\pi(\bar x)$, and
        \item $z_{(-\infty,n]}=x_{(-\infty,n]},\, z_{[i,\infty)}=\bar x_{[i,\infty)}$ for some $i\geq n$.
    \end{enumerate}
\end{defn}

We write $x\sim \bar x$ and say $x$ and $\bar x$ are in the same \emph{transition class} if $x\ra \bar x$ and $\bar x\ra x$. Then the relation $\sim$ is an equivalence relation. Denote the set of transition classes in $X$ over $y\in Y$ by $\mathscr{C}(y)$. We say there is a transition from a class $[x]$ to another class $[\bar x]$ and denote it by $[x]\ra [\bar x]$ if $x\ra \bar x$. Note that if $[x]\ra [\bar x]$ then for each $z \sim x$ and $\bar z \sim \bar x$ we have $z \ra \bar z$.

\begin{fact}\cite{AllQ13}
Let $\pi:X\ra Y$ be a one-block factor code from a one-step shift of finite type $X$ onto an irreducible sofic shift $Y$. Then
   \begin{enumerate}
       \item $|\mathscr{C}(y)|<\infty$ for each $y$ in $Y$.
       \item Let $x,x'\in \piinv(y)$ for some $y\in Y$. Given $x_{a_i}=x'_{a_i}$ where
       $(a_i)_{i\in\mathbb N}$ is a strictly increasing sequence in $\Z$, we have $x\sim x'$.
    \end{enumerate}
\end{fact}

\begin{defn}\label{defn:class_degree}
    Let $(X,Y,\pi)$ be a factor triple. The minimal number of transition classes over points of $Y$ is called the \emph{class degree} of $\pi$ and is denoted by $c_\pi$.
\end{defn}

It is clear that $c_\pi$ is invariant under conjugacy. It was shown in \cite{AllQ13} that for a finite-to-one factor triple $(X,Y,\pi)$ with $Y$ irreducible, we have $c_\pi=d_\pi$.

\begin{thm}\label{thm:classdegree}\cite{AllQ13}
    Let $(X,Y,\pi)$ be a factor triple with $Y$ irreducible. Then every right transitive point of $Y$ has exactly $c_{\pi}$ transition classes.
\end{thm}


Theorem \ref{thm:con} states the class degree of a factor code in terms of another quantity which is defined concretely in terms of blocks.
\begin{defn}\label{defn:TB}
    Let $(X,Y,\pi)$ be a factor triple with $X$ one-step and $\pi$ one-block. Let $w = w_{[0,p]} \in \B_{p+1}(Y)$. Also let $n$ be an integer in $(0,p)$ and $M$ be a subset of $\pi^{-1}(w_n)$.
    We say a block $u \in \piinv(w)$ is \emph{routable through} $a\in M$ \emph{at time} $n$ if there is a block $\bar u \in \piinv(w)$ with $\bar u_0=u_0, \bar u_{p}=u_{p}$ and $\bar u_{n}=a$. A triple $(w,n,M)$ is called a \emph{transition block} of $\pi$ if every block in $\piinv(w)$ is routable through a symbol of $M$ at time $n$. The cardinality of the set $M$ is called the \emph{depth} of the transition block $(w,n,M)$. When there is no confusion, for example when $y\in Y$ and $w=y_{[i,i+p]}$ are fixed, we say the points $x, \bar x\in\piinv(y)$ are routable through $a\in M$ \emph{at time} $i+n$ if $x_{[i,i+p]}$ and $\bar x_{[i,i+p]}$ are routable through $a$ \emph{at time} $i+n$.
\end{defn}

\begin{defn}\label{defn:cStar}
    Let $$c^*_{\pi}=\min\{|M|\colon (w,n,M)\textrm{ is a transition block of $\pi$}\}.$$
    A \emph{minimal transition block} of $\pi$ is a transition block of depth $c^*_\pi$.
\end{defn}

\begin{thm}\label{thm:con}\cite{AllQ13}
    Let $(X,Y,\pi)$ be a factor triple with $X$ one-step, $\pi$ one-block and $Y$ irreducible. Then $c_\pi=c^*_\pi$.
\end{thm}

For more details on symbolic dynamics, see \cite{LM}. For a perspective on the class degree and its relation to the degree, see \cite{AllQ13}.

\vspace{0.3cm}

\section{Each transition class over a right transitive point \\ contains a right transitive point}
\newcommand{\X}{\mathsf{X}}

In this section, we prove that given an irreducible factor triple $(X,Y,\pi)$, each transition class over a right transitive point contains a right transitive point. This result can be seen as an analogue of the well-known fact that for a finite-to-one irreducible factor triple $(X,Y,\pi)$, every preimage of a right (resp. doubly) transitive point is a right (resp. doubly) transitive point. We begin with the following definition.

\begin{defn}\label{defn:(X,v)_diamond}
    Let $(X,Y,\pi)$ be a factor triple and let $\bar X$ be a proper subshift of $X$ with $\pi(\bar X)=Y$. Let $\bar v$ be in $\B(X) \setminus \B(\bar X)$. We say two blocks $u$ and $v$ in $\B(X)$ form an \emph{$(\bar X,\bar v)$-diamond} if the following hold:
    \begin{enumerate}
        \item $\pi(u)=\pi(v)$,
        \item $\bar v$ is a subblock of $v$,
        \item $u$ occurs in $\bar X$, and
        \item $u$ and $v$ share the same initial symbol and the same terminal symbol.

\end{enumerate}
\end{defn}

The following lemma is a slightly stronger version of \cite[Proposition 3.1]{Yoo11}. We include a different proof here, which is also more direct than the original one. Later in \S \ref{sec:applications} a further strengthened version of Lemma \ref{lem:modified_yoo} will be provided.
\begin{lem}\label{lem:modified_yoo}
    Let $(X,Y,\pi)$ be an irreducible factor triple with $X$ one-step and $\pi$ one-block. Let $\bar X$ be a proper subshift of $X$ with $\pi(\bar X)=Y$. Then for each block $\bar v$ in  $\B(X) \setminus \B(\bar X)$ there is an $(\bar X,\bar v)$-diamond.
\end{lem}
\begin{proof}
    For $w$ in $\mathcal{B}(Y)$ define $n(w)$ to be the maximal number of mutually separated preimages of $w$ in $\mathcal{B}(\bar X)$ and let $n$ be the infimum of $n(w)$ where $w$ runs over all the nonempty words in $\mathcal{B}(Y)$. Clearly $n$ is a positive integer. Let $w$ be a block in $\mathcal{B}(Y)$ with $n(w)=n$ and let $u$ be a preimage of $w$ in $\mathcal{B}(\bar X)$.

    Let $\bar v$ be in $\mathcal{B}(X)\setminus\mathcal{B}(\bar X)$. Since $X$ is irreducible, there is a cycle $\alpha$ in $\mathcal{B}(X)$ such that $\alpha=u\gamma\bar v\eta$, for some blocks $\gamma,\,\eta$. Denote the length of $\alpha$ by $l$, and consider $\alpha^{n+1}$. There are at least $n$ mutually separated blocks $\beta^{(1)},\cdots,\beta^{(n)}$ in $\mathcal{B}(\bar X)$ all projecting to $\pi(\alpha^{n+1})$. Note that for all $0\le j<n+1$ and $1\le m\le n$ we have $\pi(\alpha^{n+1}_{[jl,jl+|w|)})=\pi(\beta^{(m)}_{[jl,jl+|w|)})=w$. Moreover, for all $1\le m,m'\le n$ where $m\neq m'$, the blocks $\beta^{(m)}_{[jl,jl+|w|)}$ and $\beta^{(m')}_{[jl,jl+|w|)}$ are mutually separated.

    Since $n(w)=n$, for each $0\le j<n+1$ there is $1\le m_j\le n$ such that $\beta^{(m_j)}_{[jl,jl+|w|)}$ meets $u$; i.e, there is $0<i<|w|$ such that $\beta^{(m_j)}_{jl+i}=u_i$. Thus by Pigeonhole principle there is $1\le m\le n$ such that $\beta^{(m)}$ meets $u$ twice; say at positions $jl+i$ and $j'l+i'$ for some $0\le j<j'<n+1$ and $0\le i,i'<|w|$. It is clear that the blocks $\beta^{(m)}_{[jl+i,j'l+i']}$ and $\alpha^{n+1}_{[jl+i,j'l+i']}$ form an $(\bar X,\bar v)$-diamond.
\end{proof}

\begin{rem}
 Note that Lemma~\ref{lem:modified_yoo} is not necessarily true when $X$ is reducible. For example let $X$ be the orbit closure of the point $a^\infty.b^\infty$ and $Y = \{ 0^\infty \}$ and consider the trivial map $\pi : X \to Y$. Let $\bar X=\{a^\infty \}$ and $\bar v=b$.
\end{rem}

\begin{thm}\label{thm:contain_transition_pt}
    Let $(X,Y,\pi)$ be an irreducible factor triple and let $y$ in $Y$ be right transitive. Then each transition class over $y$ contains a right transitive point.
\end{thm}
\begin{proof}
    We may assume that $X$ is one-step and $\pi$ is one-block. Let $C$ be a transition class over $y$ and let $x$ be in $C$. If $x$ is right transitive, we are done. So suppose that $x$ is not right transitive. Let $\bar X$ be the $\omega$-limit set of $x$; i.e.,
    \[ \bar X = \omega(x) = \{ z \in X : \exists n_i \nearrow \infty \text{ with } \sigma^{n_i}(x) \to z \}. \]
    Then we have $\bar X \subsetneq X$. Since $\pi(x) = y$ and $y$ is right transitive, it follows that $\pi(\bar X) = Y$. Now consider an enumeration $\bar v_1, \bar v_2, \cdots$ of $\B(X)$. For each $i \in \setN$ with $\bar v_i \in \B(X) \setminus \B(\bar X)$, by Lemma \ref{lem:modified_yoo} there is an $(\bar X, \bar v_i)$-diamond $(u_i,v_i)$. Note that for this $i$, $\bar v_i$ is a subblock of $v_i$, $v_i \in \B(X) \setminus \B(\bar X)$ and $u_i \in \B(\bar X)$.

    For each $i \in \setN$, define a block $w_i \in \B(\bar X)$ by
    \[
        {w_i} =
            \begin{cases}
                \bar v_i   & \text{if $\bar v_i \in \B(\bar X)$ }   \\
                u_i         & \text{otherwise.}   \\
            \end{cases}
    \]
    Then since each $w_i$ is in $\B(\bar X)$, we can find an increasing sequence $\{n_i\}_{i=1}^\infty$ such that $x_{[n_i, n_i + |w_i|)} = w_i$ and $n_{i+1} > n_i + |w_i|$ for all $i \in \setN$. Finally, define a new point $z \in X$ obtained from $x$ by substituting each occurrence of $u_i$ at the coordinates $x_{[n_i, n_i + |w_i|)}$ for all $i \in \setN$ with $v_i \in \B(X) \setminus \B(\bar X)$. Since each $(u_i,v_i)$ forms a diamond, $z$ is indeed a point in $X$ and we have $\pi(z) = y$.

    Since there are infinitely many positive coordinates $j$ for which $x_j = z_j$, we have $z \sim x$ and therefore $z \in C$. Also, since each block in $X$ occurs infinitely many times as a subblock in the enumeration $\bar v_1, \bar v_2, \cdots$, it follows that $z$ contains all the $\bar v_i$'s and therefore is a right transitive point, as desired.
\end{proof}

\begin{rem}\label{rem:contain_lefttransition_pt}
    Let $(X,Y,\pi)$ be an irreducible factor triple and let $y$ in $Y$ be left transitive. Then each transition class over $y$ contains a left transitive point.
\end{rem}

The proof of this remark is similar to Theorem~\ref{thm:contain_transition_pt}, but much simpler. Let $\bar X$ be the $\alpha$-limit set of $x$; i.e.,
    $\bar X = \{ z \in X : \exists n_i \searrow-\infty \text{ with } \sigma^{n_i}(x) \to z \}.$ If $x$ is not left transitive, then we may construct a point $z \in X$ similarly as in the proof of Theorem~\ref{thm:contain_transition_pt}: First find a decreasing subsequence $\{n_i\}$ such that $x_{[n_i, n_i + |w_i|)} = w_i$ (same $w_i$ defined in Theorem~\ref{thm:contain_transition_pt})  and $n_{i+1} < n_i - |w_{i+1}|$ for all $i \in \setN$, and then define $z$ by substituting each occurrence of $u_i$ at the coordinates $x_{[n_i, n_i + |w_i|)}$ with $v_i$ for all $i \in \setZ$. Then $z$ is left transitive. Since $z$ and $x$ are right asymptotic, we have $z \sim x$.


The following corollary is an immediate result of Theorem~\ref{thm:contain_transition_pt} and Remark~\ref{rem:contain_lefttransition_pt}.
\begin{cor}\label{cor:contain_dbly_transition_pt}
    Let $(X,Y,\pi)$ be an irreducible factor triple and let $y$ in $Y$ be doubly transitive. Then each transition class over $y$ contains a doubly transitive point.
\end{cor}

With the following corollary, we see that for an irreducible factor triple, the cardinalities of the transition classes over right transitive points fall into two categories: They are all finite (if a factor code is finite-to-one) or all uncountable (if it is infinite-to-one).
\begin{cor}
    Let $(X,Y,\pi)$ be an irreducible factor triple and let $y$ in $Y$ be right transitive. If $\pi$ is infinite-to-one, then the cardinality of each transition class over $y$ is uncountable.
\end{cor}
\begin{proof}
    We may assume that $X$ is one-step and $\pi$ is one-block. Recall that $\pi$ is infinite-to-one if and only if it has a diamond, say, $(u,v)$ \cite[Theorem 8.1.16]{LM}. If $C$ is a transition class over $y$, there is a right transitive point $x$ in $C$ by Theorem \ref{thm:contain_transition_pt}. Then $u$ occurs infinitely many times to the right in $x$. Any  point made by replacing some occurrences of $u$ with $v$ is equivalent to $x$, which implies that $C$ is uncountable.
\end{proof}

\vspace{0.3cm}
\section{Mutual separatedness for transition classes}
If $(X,Y,\pi)$ is an irreducible finite-to-one factor triple of degree $d$, $X$ is one-step and $\pi$ is one-block, then for each doubly transitive point $y \in Y$ the set of the preimages of $y$ consists of $d$ mutually separated points in $X$. This result is one of the important properties of fibers of finite-to-one factor codes, since it is used to prove that the degree indeed equals the number combinatorially defined using a magic block~\cite[\S 9]{CovP74, Hed69, LM}.

In this section, we present a similar mutual separatedness property for transition classes: Given two points from two distinct transition classes over a right transitive point, they are mutually separated (Theorem \ref{thm:mutually_separated}). As an application, we show that there is no transition between distinct transition classes over a right transitive point, answering a question raised by Quas.

\begin{lem}\label{lem:unique_routability}
  Let $(X,Y,\pi)$ be an irreducible factor triple with $X$ one-step and $\pi$ one-block. Given a minimal transition block $(w,n,M)$, any preimage of $w$ is routable through a unique symbol of $M$.
\end{lem}
\begin{proof}
    Let $u$ be in $\pi^{-1}(w)$ and let $d=c_\pi$ be the class degree of $\pi$. If $d=1$ then the result is trivial, so suppose $d\geq 2$. Assume that $u$ is routable through two different members $a^{(1)}$ and $a^{(2)}$ of $M=\{a^{(1)},a^{(2)},\cdots,a^{(d)}\}$. Let $x$ be a point of $X$ such that $u$ occurs infinitely many times to the right, say at positions $\{[i_j,i_j+|w|)\}_{j \in \mathbb N}$ where $i_{j+1}>i_j+|w|$.

    From each transition class in $\C(\pi(x))\setminus\{[x]\}$ choose one point and denote them by $x^{(1)}, \cdots, x^{(d-1)}$. Each of these points is routable through at least one member of $M$ at time $i_j+n$ for each $j\in\mathbb N$. If there is a point in $\{x^{(1)},\cdots,x^{(d-1)}\}$ which is routable through $a^{(1)}$ or through $a^{(2)}$ at $i_j+n$ for infinitely many $j$'s, then such a point is equivalent to $x$ which gives a contradiction and we are done.

    Suppose there is no such points; i.e., each of the points $x^{(1)}, \cdots, x^{(d-1)}$ is routable through a symbol in $\{a^{(3)},\cdots,a^{(d)}\}$ at $i_j+n$ for all but finitely many $j$'s. It follows, by Pigeonhole principle, that there are at least two points in $\{x^{(1)}, \cdots, x^{(d-1)}\}$ which are routable through the same symbol in $\{a^{(3)},\cdots,a^{(d)}\}$ at $i_j+n$ for infinitely many $j$. This forces such two points to be equivalent, which is again a contradiction.
\end{proof}
\begin{rem}
 Note that Lemma~\ref{lem:unique_routability} is not necessarily true when $X$ is reducible. For example let $X$ be the orbit closure of the point $a^\infty.b^\infty$ and $Y = \{ 0^\infty \}$ and consider the trivial map $\pi : X \to Y$. The triple $(000,1,\{a,b\})$ is a minimal transition block with a preimage $abb$ which is routable through both $a$ and $b$.
\end{rem}

\begin{lem}\label{lem:not_mutsep_implies_transition}
    Let $(X,Y,\pi)$ be an irreducible factor triple with $X$ one-step and $\pi$ one-block. Suppose there are two points $x$ and $\bar x$ such that
    \begin{enumerate}
        \item $x$ and $\bar x$ are in two distinct transition classes over a right transitive point $y \in Y$, and
        \item $x$ and $\bar x$ are not mutually separated.
    \end{enumerate}
    Then given a right transitive point $z$ in $X$, there is a transition from $[z]$ to a transition class over $\pi(z)$ other than $[z]$.
\end{lem}

\begin{proof}
Let $(w,n,M)$ be a minimal transition block. Since $y$ is right transitive, there is a sequence $\{i_j\}_{j \in \setN}$ such that $y_{[i_j,i_j+|w|)}=w$ and $i_{j+1} > i_j + |w|$. By assumption (2) for some integer $i$ we have $x_i=\bar x_i$. Since $x$ and $\bar x$ are in distinct transition classes, there is $i_j>i$ and two distinct symbols $a,b\in M$ such that $x$ and $\bar x$ are routable through $a$ and $b$, respectively, at time $i_j+n$. By taking equivalent points in the transition classes of $x$ and $\bar x$, we may assume $x_{i_j+n}=a$ and $\bar x_{i_j+n}=b$.

Consider the block $y_{[i,i_j+|w|)}$ which is an extension of $w$. Denote this extension by $\bar w$, and the coordinate $i_j-i+n$ by $\bar n$. Then $(\bar w,\bar n,M)$ is also a minimal transition block. Note that by above, block $\bar w$ has two preimages $x_{[i,i+|\bar w|)}$ and $\bar x_{[i,i+|\bar w|)}$ which share the same initial symbol; moreover, $x_{i+\bar n}=a$ and $\bar x_{i+\bar n}=b$.

Now let $z$ be a right transitive point in $X$. Note that for infinitely many $k\in\Z$ we have $z_{[k,k+|\bar w|)}=x_{[i,i+|\bar w|)}$. Denote the class degree by $d$ and let $z^{(1)} = z$. From each transition class of $\C(\pi(z))\setminus\{[z]\}$ choose a point and denote them by $z^{(2)},z^{(3)},\cdots, z^{(d)}$. Since $(\bar w,\bar n,M)$ is a minimal transition block, by Lemma \ref{lem:unique_routability} we may assume that $z^{(j)}_{k+\bar n} \in M$ for each $1 < j \leq d$. Since $z^{(j)}$'s are from distinct transition classes, the set $\{ z^{(j)}_{k+\bar n} : 1 \leq j \leq d \}$ consists of $d$ symbols for all large such $k$. Thus there is a point $\bar z$ among $z^{(2)},z^{(3)},\cdots,z^{(d)}$ such that among those $k$ at infinitely many $l$ we have $\bar z_{l+\bar n}=b$. Then for any such $l$ the points $u^{(l)}$ defined by
     $$u^{(l)}_{(-\infty,l)}=z_{[-\infty,l)}, u^{(l)}_{[l,l+\bar n)}=\bar x_{[i,i+\bar n)}\textnormal{ and }u^{(l)}_{[l+\bar n,\infty)}=\bar z_{[l+\bar n,\infty)}$$ give a transition $[z]\to [\bar z]$ over the right transitive point $\pi(z)$.
\end{proof}

\begin{thm}\label{thm:mutually_separated}
    Let $(X,Y,\pi)$ be an irreducible factor triple with $X$ one-step and $\pi$ one-block. Let $y \in Y$ be right transitive. Then any two points from two distinct transition classes over $y$ are mutually separated.
\end{thm}
\begin{proof}
    Suppose not; i.e., there are two points $x$ and $\bar x$ in distinct transition classes over $y$ such that $x$ and $\bar x$ are not mutually separated. Recall that any transition class over $y$ contains a right transitive point. Then by Lemma \ref{lem:not_mutsep_implies_transition}, given any class $C\in\C(y)$ there is a transition from $C$ to some other transition class over $y$. However, since there are only finitely many classes in $\C(y)$, there must be a transition class over $y$ with no transition to any other class which is a contradiction.
\end{proof}

One can easily check the following corollary, which is a conjugacy-invariant version of Theorem \ref{thm:mutually_separated}.

\begin{cor}\label{thm:mutually_separated_conjugacy_inv}
    Let $(X,Y,\pi)$ be an irreducible factor triple. Then there is $c > 0$ such that, whenever $y \in Y$ is right transitive and $x, \bar x$ are points from two distinct transition classes over $y$, we have $d(x,\bar x) > c$.
\end{cor}

\begin{cor}
    Let $(X,Y,\pi)$ be an irreducible factor triple. Then each transition class over a right transitive point is a closed set (with respect to the usual topology on $X$).
\end{cor}
\begin{proof}
    We may assume that $X$ is one-step and $\pi$ is one-block. Let $\{x^{(i)}\}_{i\in\mathbb N}$ be a convergent sequence in a transition class $C$ over a right transitive point $y$. Denote the limit of this sequence by $x$. Then $\pi(x)=y$. Since $x^{(i)} \to x$, for large $i$ we have $x^{(i)}_0 = x_0$. Then $x$ is not mutually separated from this $x^{(i)}$, so by Theorem \ref{thm:mutually_separated} that $x$ must belong to $C$.
\end{proof}

\begin{cor}\label{thm:no_transition}
    Let $(X,Y,\pi)$ be an irreducible factor triple and let $y$ be a right transitive point in $Y$. There is no transition between any two distinct transition classes over $y$.
\end{cor}
\begin{proof}
    We may assume that $X$ is one-step and $\pi$ is one-block.
    Suppose, on the contrary, that there is a transition $[x] \to [\bar x]$ between two distinct transition classes $[x]$ and $[\bar x]$ over $y$. Then there is a point $z \in \pi^{-1}(y)$ such that $z_{(\infty,0]} = x_{(\infty,0]}$ and $z_{[i,\infty)} = \bar x_{[i,\infty)}$ for some $i > 0$. Since $z \in [\bar x]$ and $z$ and $x$ are not mutually separated, by Theorem \ref{thm:mutually_separated} we have a contradiction.
\end{proof}

The following examples show that there may be a transition if the domain is not irreducible, or the point in $Y$ is not right transitive.

\begin{ex}\label{ex:trans}
(1) Let $X$ be the orbit closure of the point $a^\infty .b^\infty$ and $Y = \{ 0^\infty \}$ and consider the trivial map $\pi : X \to Y$. Above the point $0^\infty$ there are two transition classes: one class, say $C_1$, consists of only one point $a^\infty$ and the other class $C_2$ consists of all the points of the form $a^\infty b^\infty$ and $b^\infty$. Even though $0^\infty$ is a (right) transitive point, there is a transition from $C_1$ to $C_2$. Note that $X$ is not irreducible.

(2) Let $X$ be an irreducible edge shift given by the diagram and $Y = \{ 0, 1 \}^\setZ$. Let $\pi : X \to Y$ be a one-block factor code given by $\pi(a) = \pi(b) = \pi(d) = 0$ and $\pi(c) = \pi(e) = 1$. Note that $c_\pi=1$.

    
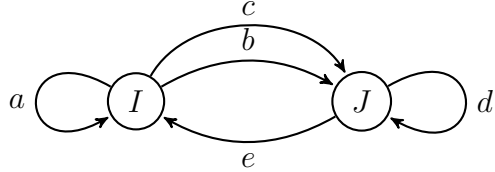
\begin{center}
    \begin{tikzpicture}[->,>=stealth',shorten >=1pt,auto,node distance=3cm, thick,main node/.style={circle,draw}]

        \node[main node] (I) {$I$};
        \node[main node] (J) [right of=I] {$J$};

        \path
            (I) edge [loop left, looseness=10, out=150, in=210] node {$a$} (I)
                edge [bend left] node {$b$} (J)
                edge [bend left, out=60, in=120] node {$c$} (J)
            (J) edge [loop right, looseness=10, out=30, in=330] node {$d$} (J)
                edge [bend left] node {$e$} (I);
    \end{tikzpicture}
    \end{center}

Let $y$ in $Y$ be any point with $y_{[0,\infty)} = 0^\infty$. Since each vertex has incoming edges labeled by $0$ and by $1$ respectively, there are two left infinite paths, say $\alpha$ and $\beta$, mapping to $y_{(-\infty,0]}$ and terminating at $I$ and $J$, respectively. Hence as in the previous example, $y$ has two transition classes: one containing $\alpha.a^\infty$ and the other containing $\alpha.a^n b d^\infty$ for all $n \in \setN$ and $\beta. d^\infty$. It follows that there is a transition from one transition class over $y$ to the other. Note that one can even take $y$ to be left transitive, so even the left transitivity of $y$ does not guarantee no transition between transition classes over $y$.
\end{ex}

In the definition of a transition $x \to \bar x$ between two points $x$ and $\bar x$ in $X$, we require an infinite number of points which are left asymptotic to $x$ and right asymptotic to $\bar x$, since for each $n \in \mathbb{N}$ we need a point which equals $x$ until time $n$. However, if the point $y = \pi(x) = \pi(\bar x)$ is right transitive, to satisfy this definition, a single asymptotic point to $x$ and $\bar x$ suffices, as condition (3) in the following corollary shows.

\begin{cor}\label{cor:transition_equals_equivalence}
    Let $(X,Y,\pi)$ be an irreducible factor triple and $y \in Y$ be right transitive. Then for $x, \bar x \in \pi^{-1}(y)$ the following are equivalent:
    \begin{enumerate}
        \item $x \sim \bar x$
        \item $x \to \bar x$
        \item There is a point in $\pi^{-1}(y)$ which is left asymptotic to $x$ and right asymptotic to $\bar x$.
    \end{enumerate}
\end{cor}
\begin{proof}
    Since all the conditions (1), (2) and (3) are conjugacy-invariant, we may assume that $X$ is one-step and $\pi$ is one-block. The equivalence of (1) and (2) follows from Corollary \ref{thm:no_transition}. (2) clearly implies (3). Suppose that $x$ and $\bar x$ are not equivalent. If $z$ is a point satisfying condition (3), we have $z \in [\bar x]$. Since $z$ and $x$ are left asymptotic, they are not mutually separated, which contradicts Theorem \ref{thm:mutually_separated}.
\end{proof}

The following corollary is immediate.
\begin{cor}
    Let $(X,Y,\pi)$ be an irreducible finite-to-one factor triple. Let $y$ in $Y$ be right transitive. Then for each $x$ and $\bar x$ in $\pi^{-1}(y)$, they are either mutually separated or right asymptotic.
\end{cor}
\begin{proof}
First note that since $\pi$ is finite-to-one, $x$ and $\bar x$ must be right transitive. Suppose $x$ and $\bar x$ are neither mutually separated nor right asymptotic. Then by Corollary \ref{cor:transition_equals_equivalence} they are equivalent and therefore produce a diamond which implies that $\pi$ is not finite-to-one (see \cite[\S8.1]{LM}).
\end{proof}

\vspace{0.3cm}
\section{Symbol partition properties for a factor triple}

In this section, we provide further separation properties for transition classes over a right transitive point.
Throughout this section, we assume that $X$ is one-step and $\pi$ is one-block for a factor triple $(X,Y,\pi)$. Also in what follows, for a transition class $C$ and $A\subset\mathbb{Z}$, denote by $C|_A$ the set $\{ x_A : x \in C \}$. If $A=\{i\}$ for some integer $i$ then $C|_A$ is denoted simply by $C|_i$ for the convenience. With this notation, Theorem \ref{thm:mutually_separated} says that if $y \in Y$ is right transitive, then $C|_i \cap \bar C|_i = \emptyset$ for $C \neq \bar C \in \C(y)$ and all $i\in\setZ$.

We will first see that if $(w,n,M)$ is a minimal transition block and $y$ is right transitive, then $|C|_{i+n}\cap M|=1$ for any $C$ in $\C(y)$ and $i\in\mathbb{Z}$ with $y_{[i,i+|w|)}=w$.

\begin{lem}\label{lem:same_member}
    Let $(X,Y,\pi)$ be an irreducible factor triple. Let $y \in Y$ be right transitive with $y_{[i,i+|w|)}=w$ for some minimal transition block $(w,n,M)$. Given $C \in \C(y)$, there is a unique symbol $b$ in $M$ such that every point in $C$ is routable through $b$ at time $i+n$.
\end{lem}

\begin{proof}
    Let $M_C=M\cap C|_{i+n}$ for any $C$ in $\C(y)$. Then $\bigcup_{C\in\C(y)}M_C=M$. As each $M_C$ is contained in $C|_{i+n}$, by Theorem \ref{thm:mutually_separated} the sets $M_C$'s are mutually disjoint. So $\sum_{C\in\C(y)}|M_C|=|M|=c_\pi$. Since $|\C(y)|=c_\pi$ and each $M_C$ is nonempty, we have that $|M_C|=1$ for each $C\in\C(y)$.
    \end{proof}

The following proposition intuitively states that given a right transitive point $y$ in $Y$, if two preimages of the symbol $y_0$ appear in two distinct transition classes over $y$ at time $0$, then these two symbols cannot ever appear at the same time in a single class over any right transitive point.

\begin{prop}\label{prop:symbols_once_separated_cannot_meet_again}
    Let $(X,Y,\pi)$ be an irreducible factor triple and let $y\in Y$ be a right transitive point with a minimal transition block in $y_{(-\infty,0)}$. Then for any right transitive point $z\in Y$ with $z_i= y_0$, $i\in\setZ$, and any given $D \in \C(z)$, we have $D|_i \cap C|_0\neq \emptyset$ for at most one $C \in \C(y)$.

\end{prop}
\begin{proof}
    Let $y_0=a$. Suppose, on the contrary, that for some right transitive point $z \in Y$ with $z_i = a$, $i\in\setZ$, and $D \in \C(z)$, there are two symbols $b \in C|_0$ and $c \in \bar C|_0$, where $C \neq \bar C \in \C(y)$, occurring at the time $i$ in $D$, that is, $b, c \in D|_i$. Then there are two points $x^{(1)}$ and $x^{(2)}$ in $D$ with $x^{(1)}_i = b$ and $x^{(2)}_i = c$. As $b\in C|_0$ and $c\in\bar C|_0$, we have two points $\bar x^{(1)} \in C$ and $\bar x^{(2)} \in \bar C$ such that $\bar x^{(1)}_0 = b$ and $\bar x^{(2)}_0 = c$.

    Define new points $\hat x^{(l)}$ for $l = 1, 2$ by
    $$ \hat x^{(l)} = \bar x^{(l)}_{(-\infty,0)}.x^{(l)}_{[i,\infty)}$$
    and consider $\hat z = \pi(\hat x^{(1)}) = \pi(\hat x^{(2)})$. Then $\hat z$ is right transitive since it is right asymptotic to $\sigma^i(z)$, and the points $\hat x^{(l)}$'s are equivalent since $x^{(l)}$'s are equivalent.

   Recall that by assumption, a minimal transition block $(w,n,M)$ occurs in $y_{[j,j+|w|)}=\hat{z}_{[j,j+|w|)}$ for some $j \leq -|w|$. Since $\bar x^{(1)}$ and $\bar x^{(2)}$ are not equivalent then by Theorem \ref{thm:mutually_separated} they are routable through different members of $M$ at time $j+n$. It follows that $\hat x^{(1)}$ and $\hat x^{(2)}$ are routable through different members of $M$ at time $j+n$ which contradicts Lemma \ref{lem:same_member} since $\hat x^{(1)}$ and $\hat x^{(2)}$ are equivalent.
\end{proof}

 A finite-to-one factor code possesses a kind of permutation property, as discussed in \cite[\S 9]{LM}, for preimages of a magic symbol. Analogous to this property, the following proposition exhibits a permutation property for preimages of a minimal transition block.
\begin{prop}\label{cor:min_transit_blk_permutation}
    Let $(X,Y,\pi)$ be an irreducible factor triple of class degree $d$ and let $(w,n,M)$ be a minimal transition block. For each block $u = wvw \in \B(Y)$ for some $v$, there is a permutation $\tau_u : M \to M$ such that given any right transitive point $y$ in $Y$ with $y_{[0,|u|)} = u$ and $C \in \C(y)$, we have $\tau_u ( C|_{n} \cap M) = C|_{|wv|+n} \cap M$.
\end{prop}
\begin{proof}

  Let $y$ be a right transitive point in $Y$ with $y_{[0,|u|)}=u$. Given $C$ in $\C(y)$, by Lemma \ref{lem:same_member}, we have $C|_n\cap M=\{a\}$ and $C|_{|wv|+n}\cap M=\{b\}$ for some symbols $a$ and $b$. Define a permutation $\tau_{u,y}$ of $M$ by $\tau_{u,y}(a)=b$, and let $x\in C$ be a point with $x_n=a$ and $x_{|wv|+n}=b$.

  We show that $\tau_{u,y}$ does not depend on $y$. Let $\bar y$ be a right transitive point in $Y$ with $\bar y_{[0,|u|)}=u$, and let $\bar x$ be a point in $\piinv(\bar y)$ with $\bar x_n=a$. Consider a point $x'\in\piinv(\bar y)$ with $x'_{|wv|+n}=b$. Then the point $\bar x_{(-\infty,n]}{x}_{(n,|wv|+n)}x'_{[|wv|+n,\infty)}\in\piinv(\bar y)$ is left asymptotic to $\bar x$ and right asymptotic to $x'$, which implies that $\bar x$ and $x'$ are equivalent. It follows that $\tau_{u,\bar y}(a)=b$, and thus $\tau_{u,y}=\tau_{u,\bar y}$ as $C$ was chosen arbitrarily.
%
\end{proof}

Theorem \ref{thm:block_partition} informally states that given any irreducible factor triple $(X,Y,\pi)$, in order to determine whether two points with the same image are equivalent or not, one only needs to locally compare the preimages of a magic block which occur within these points. More precisely, there is a partition on the set of preimages of a magic block $w$ of $\pi$ such that given all right transitive points $y$ in $Y$ and any $i\in\mathbb{Z}$ with $y_{[i,i+|w|)}=w$, two preimages $x,x'$ of $y$ are equivalent if and only if $x_{[i,i+|w|)}$ and $x'_{[i,i+|w|)}$ belong to the same class of the partition on $\piinv(w)$.

\begin{thm}\label{thm:block_partition}
    Let $(X,Y,\pi)$ be an irreducible factor triple of class degree $d$. Let $w$ be a magic block of $\pi$. There is a partition of $\pi^{-1}(w)$ into $d$ subsets $\mathcal{B}_1,\cdots,\mathcal{B}_{d}$ such that for any doubly transitive $z$ in $Y$ and any $i \in \setZ$ with $z_{[i,i+|w|)} = w$, we have a bijection $\rho_{z,i} : \C(z) \to \{1,\dots, d\}$ with $D|_{[i,i+|w|)} \subseteq\mathcal{B}_{\rho_{z,i}(D)}$ for each $D \in \C(z)$.
\end{thm}
\begin{proof}
    Let $k$ be a magic coordinate of $w$ and $\mathcal{A}_{w,k}=\{u_k\mid u\in\pi^{-1}(w)\}$. Let $y$ be a doubly transitive point in $Y$ with $y_{[0,|w|)}=w$. For $C\in\C(y)$ let
    $\mathcal{B}_C=\{u\in\pi^{-1}(w)\mid u_k\in C|_{k}\}.$ List $\C(y)=\{C^{(1)},\cdots,C^{(d)}\}$ and denote $\mathcal{B}_j=\mathcal{B}_{C^{(j)}}$. We claim that $\{\mathcal{B}_j\colon 1\leq j\leq d\}$ is a desired partition.

    Note that since $k$ is a magic coordinate of $w$, we have $\bigcup_{C \in \C(y)} C|_k = \mathcal{A}_{w,k}$ and hence $\bigcup_{C\in\C(y)}\mathcal{B}_C=\pi^{-1}(w)$. Moreover, $\mathcal{B}_C$'s with $C\in\C(y)$ are mutually disjoint, since by Theorem \ref{thm:mutually_separated} the sets $C|_{k}$'s with $C\in\C(y),$ are mutually disjoint. 
    
    Now we show that such a partition does not depend on $y$. Let $z$ be a doubly transitive point in $Y$ with $z_{[i,i+|w|)}=w$ for some $i\in\mathbb{Z}$. By considering a point $\sigma^i (z)$ we may assume that $i = 0$.
    For each $D\in\C(z)$, let $\mathcal{B}_D^{(z)}=\{u\in\pi^{-1}(w)\mid u_{k}\in D|_{k}\}$. Again since $k$ is a magic coordinate of $w$, we have $\bigcup_{D \in \C(z)} D|_k = \mathcal{A}_{w,k}$ and thus $\bigcup_{D\in\C(z)}\mathcal{B}_D^{(z)}=\pi^{-1}(w)$. It follows that for each $D \in \C(z)$ there is a transition class $C\in\C(y)$ such that $C|_k\cap D|_{k}\neq\emptyset,$ which implies that $\mathcal{B}_C\cap\mathcal{B}_D^{(z)}\neq\emptyset.$ We show that $\mathcal{B}_C=\mathcal{B}_D^{(z)}$. Suppose not. First assume that there is a block $u$ in $\mathcal{B}_C\setminus\mathcal{B}_D^{(z)}$. There must be another class $D' \in \C(z)$ such that $u_k\in D'|_{k}$. This means $C|_k$ intersects both $D|_{k}$ and $D'|_{k}$ which contradicts Proposition \ref{prop:symbols_once_separated_cannot_meet_again}. So $\mathcal{B}_C\subseteq\mathcal{B}_D^{(z)}$. The case $u\in\mathcal{B}_D^{(z)}\setminus\mathcal{B}_C$ is also impossible by the 
symmetric argument, and hence we have $\mathcal{B}_C=\mathcal{B}_D^{(z)}$.

    Define $\rho_{z,i}:\C(z)\to\{1,\cdots,d\}$ to send each $D\in\C(z)$ to the unique $1\le\rho_{z,i}(D)\le d$ with $\mathcal{B}^{(\sigma^i(z))}_{D}=\mathcal{B}_{\rho_{z,i}(D)}$. By definition, $D|_{[i,i+|w|)}$ is contained in $\mathcal{B}_{\rho_{z,i}(D)}$.
\end{proof}

Note that in general, the partition in Theorem \ref{thm:block_partition} need not be unique. However, 
in some special cases, for example when $\pi$ has a magic symbol, we obtain the uniqueness of the partition. The following remark which follows directly from Proposition \ref{prop:symbols_once_separated_cannot_meet_again} explains such cases.

\begin{rem}\label{rem:symbol-partition}
Let $(X,Y,\pi)$ be an irreducible factor triple of class degree $d$ and let $a$ be in $\A(Y)$. If for some doubly transitive point $y$ in $Y$ we have $y_0=a$ and $\piinv(y)|_0=\piinv(a)$ then there is a unique partition of $\pi^{-1}(a)$ into $d$ subsets $\B_1,\cdots,\B_{d}$ such that for any right transitive $z \in Y$ with $z_i = a$, we have a bijection $\rho_{z,i}: \C(z) \to \{ 1, \cdots, d \}$ with $C|_i \subseteq\B_{\rho_{z,i}(C)}$ for each class $C$ in $\C(z)$. 
\end{rem}
The following example shows that the partition property stated in Remark~\ref{rem:symbol-partition} need not hold for every symbol.

\def\vrtd{1.0cm}\def\hrzd{2cm}

\begin{figure}
  \begin{tikzpicture}[->,>=stealth',shorten >=1pt,auto,on grid,semithick,inner sep=2pt,bend angle=45,state/.style]
    \node[state] (m1) {$m_1$};
    \node[state] (alp0) [above right=2*\vrtd and \hrzd of m1] {$\alpha_0$};
    \node[state] (alp1bet1) [below=\vrtd of alp0] {};
    \node[black] (a1b1) at (alp1bet1) {$\alpha_1+\beta_1$};
    \node[state] (bet2gmm2) [below=\vrtd of alp1bet1] {};
    \node[black] (b2g2) at (bet2gmm2) {$\beta_2+\gamma_2$};
    \node[state] (gmm3) [below=\vrtd of bet2gmm2] {$\gamma_3$};
    \node[state] (0) [right=\hrzd of alp0] {$0$};
    \node[state] (1) [right=\hrzd of alp1bet1] {$1$};
    \node[state] (2) [right=\hrzd of bet2gmm2] {$2$};
    \node[state] (3) [right=\hrzd of gmm3] {$3$};
    \node[state] (gmm0') [right=\hrzd of 0] {$\gamma_0'$};
    \node[state] (alp1'bet1') [below=\vrtd of gmm0'] {};
    \node[black] (a1b1') at (alp1'bet1') {$\alpha_1'+\beta_1'$};
    \node[state] (bet2'gmm2') [below=\vrtd of alp1'bet1'] {};
    \node[black] (b2g2') at (bet2'gmm2') {$\beta_2'+\gamma_2'$};
    \node[state] (gmm3') [below=\vrtd of bet2'gmm2'] {$\gamma_3'$};
    \node[state] (m2) [right=\hrzd of gmm0'] {$m_2$};
    \node[state] (m1') [right=\hrzd of bet2'gmm2'] {$m_1$};

    \node[state] (X) [below left=2*\vrtd and \hrzd of m1] {$X:$};

    \node[state] (m2') [below=4*\vrtd of m1] {$m_2$};
    \node[state] (alp3) [below=\vrtd of gmm3] {$\alpha_3$};
    \node[state] (alp4) [below=\vrtd of alp3] {$\alpha_4$};
    \node[state] (bet5) [below=\vrtd of alp4] {$\beta_5$};
    \node[state] (gmm6) [below=\vrtd of bet5] {$\gamma_6$};
    \node[state] (gmm7) [below=\vrtd of gmm6] {$\gamma_7$};
    \node[state] (4) [right=\hrzd of alp4] {$4$};
    \node[state] (5) [right=\hrzd of bet5] {$5$};
    \node[state] (6) [right=\hrzd of gmm6] {$6$};
    \node[state] (7) [right=\hrzd of gmm7] {$7$};
    \node[state] (alp3') [below=\vrtd of gmm3'] {$\alpha_3'$};
    \node[state] (bet4') [below=\vrtd of alp3'] {$\beta_4'$};
    \node[state] (alp5'bet5'gmm5') [below=\vrtd of bet4'] {};
    \node[black] (a5b5g5') at (alp5'bet5'gmm5') {$\beta_5'+\alpha_5'+\gamma_5'$};
    \node[state] (bet6'gmm6') [below=\vrtd of alp5'bet5'gmm5'] {};
    \node[black] (b6g6') at (bet6'gmm6') {$\beta_6'+\gamma_6'$};
    \node[state] (alp7') [below=\vrtd of bet6'gmm6'] {$\alpha_7'$};
    \node[state] (m2_) [right=\hrzd  of a5b5g5'] {$m_2$};
    \node[state] (m1_) [right=\hrzd of alp7'] {$m_1$};

    \node[state] (Y) [below left=4*\vrtd and \hrzd of m2'] {$Y:$};

    \node[state] (m) [below=4*\vrtd of m2'] {$m$};
    \node[state] (alpbetgmm) [right=\hrzd of m] {};
    \node[black] (abg) at (alpbetgmm) {$\alpha+\beta+\gamma$};
    \node[state] (a) [right=\hrzd of abg] {$a$};
    \node[state] (alp'bet'gmm') [right=\hrzd of a] {};
    \node[black] (abg') at (alp'bet'gmm') {$\alpha'+\beta'+\gamma'$};
    \node[state] (m') [right=\hrzd of abg'] {$m$};

    \draw (0.5*\hrzd,-6*\vrtd) to [line to] node {$\pi$} (0.5*\hrzd,-7.5*\vrtd);

    \path
      (m1) edge (alp0) edge (a1b1) edge (b2g2) edge (gmm3)
      (alp0) edge (0)
      (a1b1) edge (1)
      (b2g2) edge (2)
      (gmm3) edge (3)
      (0) edge (gmm0')
      (1) edge (a1b1')
      (2) edge (b2g2')
      (3) edge (gmm3')
      (gmm0') edge (m2)
      (a1b1') edge (m1')
      (b2g2') edge (m1')
      (gmm3') edge (m1')

      (m2') edge (alp3) edge (alp4) edge (bet5) edge (gmm6) edge (gmm7)
      (alp3) edge (3)
      (alp4) edge (4)
      (bet5) edge (5)
      (gmm6) edge (6)
      (gmm7) edge (7)
      (3) edge (alp3')
      (4) edge (bet4')
      (5) edge (a5b5g5')
      (6) edge (b6g6')
      (7) edge (alp7')
      (alp3') edge (m2_)
      (bet4') edge (m2_)
      (a5b5g5') edge (m2_)
      (b6g6') edge (m2_)
      (alp7') edge (m1_)

      (m) edge (abg)
      (abg) edge (a)
      (a) edge (abg')
      (abg') edge (m');
  \end{tikzpicture}
  \caption{Graph of Example~\ref{ex:no-partition}}
\end{figure}
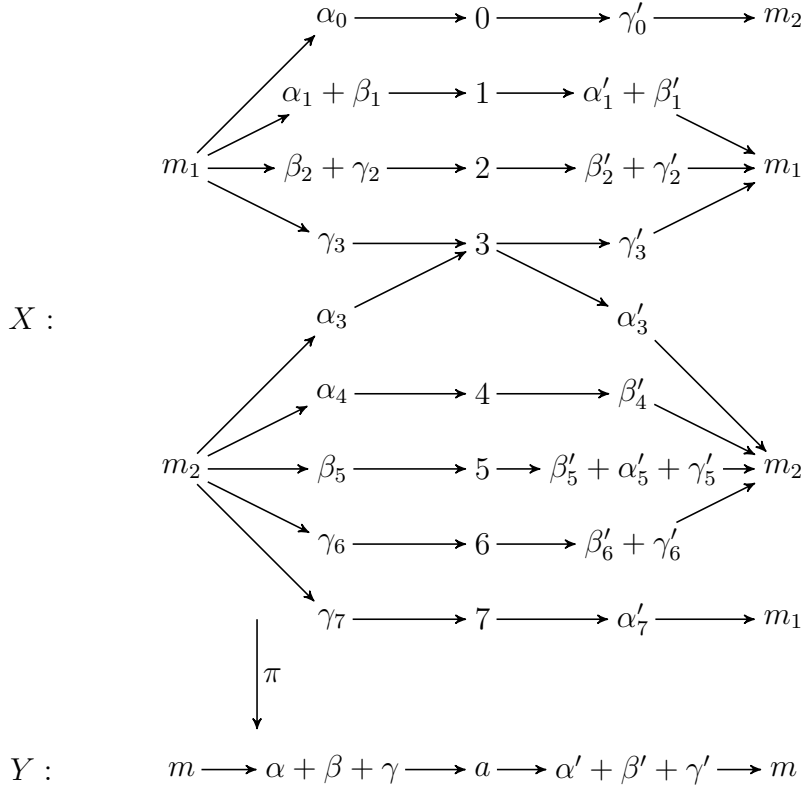

\begin{ex}\label{ex:no-partition}
    Consider the irreducible vertex shifts $X$ and $Y$ displayed in Figure 1. In the graphs of $X$ and $Y$, each pair of leftmost and rightmost vertices with the same symbol ($m_1, m_2$ and $m$) is identified. Let $\pi$ be the map erasing the subscripts, and sending all numbered vertices of $X$ to the symbol $a$.

    A point $y$ in $Y$ is of the form $(m (\alpha + \beta + \gamma) a (\alpha' + \beta' + \gamma') )^\infty$ where the notation $\alpha + \beta + \gamma$ implies that one can choose freely either $\alpha$ or $\beta$ or $\gamma$ as a symbol to appear in the given position (this is a standard notion in the theory of formal languages). The choices may differ up to positions so that a word like $m\alpha a\beta' m\beta a \gamma'$ is allowed. If $ u = m \delta a \delta' m$ with $\delta \in \{\alpha, \beta, \gamma\}$ and $\delta' \in \{ \alpha', \beta', \gamma' \}$, this $u$ defines a \emph{unique permutation} $\tau_u$ of $\{m_1, m_2 \}$: If $m_1$ and $u$ are given, for any preimage $v$ of $u$ with $v_0 = m_1$, the last symbol of $v$ is uniquely determined. Similarly for $m_2$. Note that this property is inductively true for general blocks of the form $u = m \delta^{(1)} a \delta'^{(1)} m \cdots m \delta^{(k)} a \delta'^{(k)} m$. This means that for each $y$ in $Y$ and $x$ in $\pi^{-1}(y)$ there are only two 
choices of putting $m_i$'s in $x$. Hence, there are exactly two transition classes over $y$. It follows that every point in $Y$ has 2 transition classes and the class degree of $\pi$ is 2. Note that $m$ is a magic symbol, and also, it is a minimal transition block.

    Now consider the symbol $a \in \mathcal{A}(Y)$ and consider the sets $C|_0$ with $C \in \C(y)$, with $y$ right or doubly transitive.
\begin{enumerate}
    \item If $y = \cdots m\alpha .a \alpha'm \cdots$, then we have $\{1\}$ and $\{3\}$ as $C|_0$'s in Proposition \ref{prop:symbols_once_separated_cannot_meet_again}.
    \item If $y = \cdots m\beta .a \beta'm \cdots$, then we have $\{1, 2\}$ and $\{5\}$.
    \item If $y = \cdots m\gamma .a \gamma'm \cdots$, then we have $\{2, 3\}$ and $\{6\}$.
    \item If $y = \cdots m\delta .a \delta'm \cdots$, where $(\delta, \delta')$ is not a pair in the above cases, then we have each $C|_0$ is a singleton.
\end{enumerate}
Symbols $1,2$ appear together in one class, as do symbols $2,3$; however, symbols $1,3$ appear in distinct classes. So, there is no partition of $\piinv(a)$ into 2 sets satisfying the property stated in Remark \ref{rem:symbol-partition}.
\end{ex}

\vspace{0.2cm}
\section{Applications and examples}\label{sec:applications}

The definitions of transition, transition classes, and the class degree are asymmetric. So it is natural to consider reversed transition as follows: $x \to_r \bar x$ if for each integer $n$, there is a point which is left asymptotic to $x$ and equal to $\bar x$ in $[n,\infty)$. We say that $x \sim_r \bar x$ if and only if $x \to_r \bar x$ and $\bar x \to_r x$. The \emph{reversed transition classes} of a point in $Y$ are the equivalence classes made by this new relation $\sim_r$. Denote by $[x]_r$ the reversed transition class containing $x$. Let $x^\tr$ be the point that $(x^\tr)_i = x_{-i}$ for every integer $i$ and define $X^\tr = \{x^\tr : x \in X\}$. Note that the reversed transition classes of a point $y$ can be obtained from the transition classes of a point $y^\tr$ under the \emph{transposed code} $\pi^\tr : X^\tr \to Y^\tr$. All the results in the previous sections hold for the reversed transition classes if we replace left transitive with right transitive.

Note that the set of transition classes and that of the reversed transition classes of a point $y \in Y$ need not coincide, nor do they need to have the same cardinality. However, they do coincide for almost all points (Proposition~\ref{prop:reversed_transition_class_is_same_to_transition_class}).

\begin{ex}\label{ex:reversed}
    We show that given any finite-to-one irreducible factor triple $(X,Y,\pi)$ which is not bi-closing, there are uncountably many points in $Y$ for which the number of transition classes differs from the number of reversed transition classes.

    Suppose, without loss of generality, that $\pi$ is not \emph{left closing}; i.e., there are two distinct points $x\neq \bar x \in X$ which are right asymptotic and $y = \pi(x) = \pi(\bar x)$. Since the subshift $X$ is of finite type and $x_{[i,\infty)}= \bar x_{[i,\infty)}$ for some $i$, by changing the common right tail we may assume that $x$ and $\bar x$ are right transitive. Then $y$ is also right transitive and hence $y$ has exactly $c_\pi$ transition classes. Since $\pi$ does not have any diamond, $x$ and $\bar x$ are not equivalent with respect to the reversed transition relation. Moreover, since any given two points $z$ and $\bar z$ from distinct transition classes over $y$ are mutually separated, they are also not equivalent with respect to the reversed transition relation. It follows that the number of reversed transition classes of $y$ is at least $c_\pi+1$. Since the right tail of $x$ can be changed in uncountably many ways in order to produce a right transitive point, we have uncountably many 
points in $Y$ with a different number of transition classes than the number of its reversed transition classes.
\end{ex}

\begin{prop}\label{prop:reversed_transition_class_is_same_to_transition_class}
Let $(X,Y,\pi)$ be an irreducible factor triple. If $y$ is doubly transitive and $x \in \pi^{-1}(y)$, then the transition class of $x$ equals the reversed transition class of $x$.
\end{prop}
\begin{proof}
    Suppose, on the contrary that $[x] \neq [x]_r$. We may assume that there is a point $\bar x \in [x]_r \setminus [x]$. Since $\bar x \sim_r x$, there is a point $z \in \pi^{-1}(y)$ which is left asymptotic to $\bar x$ and right asymptotic to $x$. Then by Corollary \ref{cor:transition_equals_equivalence}, $[x] = [z] = [\bar x]$ and therefore $\bar x \in [x]$, which is a contradiction.
\end{proof}

Recall Definition \ref{defn:(X,v)_diamond}: Let $(X,Y,\pi)$ be a factor triple and $\bar X$ a proper subshift of $X$ with $\pi(\bar X)=Y$. Let $\bar v$ be in $\B(X) \setminus \B(\bar X)$. A block $u$ in $\mathcal{B}(\bar X)$ and a block $v$ in $\B(X)$ form an \emph{$(\bar X,\bar v)$-diamond} if $\pi(u)=\pi(v)$, $\bar v$ is a subblock of $v$, and $u$ and $v$ share the same initial symbol and the same terminal symbol. Lemma \ref{lem:modified_yoo} states that if $X$ is irreducible, one-step and $\pi$ is one-block then for each block $\bar v$ in  $\B(X) \setminus \B(\bar X)$ there is an $(\bar X,\bar v)$-diamond.

As mentioned before, we strengthen Lemma \ref{lem:modified_yoo} in the present section. Proposition \ref{prop:modified_yoo_with_upper_bound} gives an upper bound for the length of $(\bar X,\bar v)$-diamond of a given word $\bar v$ in $\mathcal{B}(X)\setminus\mathcal{B}(\bar X)$. Note that the proof of Proposition \ref{prop:modified_yoo_with_upper_bound} employs Theorem \ref{thm:contain_transition_pt} which was shown using Lemma \ref{lem:modified_yoo}.

\begin{prop}\label{prop:modified_yoo_with_upper_bound}
    Let $(X,Y,\pi)$ be an irreducible factor triple with $X$ one-step and $\pi$ one-block. Let $\bar X$ be a proper subshift of $X$ with $\pi(\bar X)=Y$. Then there is a positive integer $N$ such that for each block $\bar v$ in  $\mathcal{B}(X)\setminus\mathcal{B}(\bar X)$ we have an $(\bar X,\bar v)$-diamond of length less than $|\bar v| + N$.
\end{prop}
\begin{proof}

    Let $k$ be a positive integer such that for any blocks $u,v$ in $\mathcal{B}(X)$ there is a block $w$ in $\mathcal{B}(X)$ with $|w|\le k$ and $uwv$ in $\mathcal{B}(X)$. Such $k$ exists since $X$ is irreducible and of finite type. Let $(w,n,M)$ be a minimal transition block of $\pi$ with $|w|=l$ and $u$ a preimage of $w$ in $\mathcal{B}(\bar X)$.

    Consider a block $\bar v$ in $\mathcal{B}(X)\setminus\mathcal{B}(\bar X)$, and let $\gamma=u\alpha\bar v\beta u$ for some $\alpha$ and $\beta$ with $|\alpha|,|\beta|\le k$. Denote $|\gamma|=L$. Let $y$ be a right transitive point of $Y$ and $\bar x$ a preimage of $y$ in $\bar X$. By Theorem \ref{thm:contain_transition_pt} there is a right transitive preimage $x$ of $y$ in $X$ which is equivalent to $\bar x$. Note that $x\ne\bar{x}$. For convenience, let $x_{[0,L)}=\gamma$. Note that $x$ and $\bar x$ are both routable through the same symbol of $M$, say $a$, at time $n$, and at time $L-l+n$. Let $\delta$ be a block of length $l$ in $\mathcal{B}(X)$ such that $\delta_0=u_0,\delta_{l-1}=u_{l-1}$ and $\delta_n=a$. Also let $\bar\delta$ and $\bar\delta'$ be blocks of length $l$ in $\mathcal{B}(X)$ such that $\bar\delta_0=\bar x_0,\bar\delta_{l-1}=\bar x_{l-1},\bar\delta_n=a$ and $\bar\delta'_0=\bar x_{L-l},\bar\delta'_{l-1}=\bar x_{L-1},\bar\delta'_n=a$. Then the two blocks $\bar\delta_{[0,n)}\delta_{[n,l)}x_{
[l,L-l)}\delta_{[0,n)}\bar\delta'_{[n,l)}$ and $\bar x_{[0,L]}$ form an $(\bar X,\bar v)$-diamond of
length smaller than or equal to $|\bar v|+2l+2k$. Letting $N=2l+2k$ completes the proof.
\end{proof}

In the case of finite-to-one factor codes $\pi = \pi_2 \circ \pi_1$, we have $c_\pi=d_\pi = d_{\pi_1} \cdot d_{\pi_2}=c_{\pi_1} \cdot c_{\pi_2}$. Since class degree is a conjugacy invariant generalization of degree, it is natural to consider whether this equality holds for the infinite-to-one case. The following example shows that it actually does not; however we are still able to get an inequality as says Proposition \ref{prop:inequality}.

\begin{ex}\label{ex:multiplication_fails}
    Let $X$ be the full 2-shift and $Y = \{ 0^\infty \}$ and consider the trivial map $\pi : X \to Y$. By letting $\pi_2 = \pi$ and $\pi_1 : X \to X$ by $\pi_1(x)_i = x_i + x_{i+1} \mod 2$, we have $\pi = \pi_2 \circ \pi_1$. However, $1 = c_\pi < c_{\pi_1} \cdot c_{\pi_2} = 2$.

\end{ex}

\begin{prop}\label{prop:inequality}
  Let $(X,Y,\pi_1)$ and $(Y,Z,\pi_2)$ be irreducible factor triples. If $\pi = \pi_2 \circ \pi_1$, then $c_{\pi} \le c_{\pi_1} \cdot c_{\pi_2}$.
\end{prop}
\begin{proof}
    Since class degree is invariant under conjugacy, we may assume $X$ and $Y$ are one-step, and $\pi_1$ and $\pi_2$ (hence $\pi$) are one-block. For convenience, rename $c_1=c_{\pi_1}$ and $c_2=c_{\pi_2}$. Fix a doubly transitive point $z$ in $Z$ and let $C$ be a transition class over $z$ with respect to $\pi_2$. By Corollary \ref{cor:contain_dbly_transition_pt}, $C$ contains a doubly transitive point $y$. Moreover, by the same corollary there are $c_1$ doubly transitive points $x^{(1)},\cdots,x^{(c_1)}$ in $\pi_1^{-1}(y)$ which are not equivalent to each other with respect to $\pi_1$.

    We claim that any doubly transitive point $x'$ in $\pi_1^{-1}(C)$ is equivalent to some $x^{(i)},1\le i\le c_1$. To show the claim, observe that $y'=\pi_1(x')$ lies in $C$ and is equivalent to $y$ with respect to $\pi_2$. It follows that there is a point $\vec y$ in $Y$ such that $\pi_2(\vec y)=z$ and $\vec y_{(-\infty,0]}=y_{(-\infty,0]},\vec y_{[j,\infty)}=y'_{[j,\infty)}$ for some $j>0$.

    Since $y'$ is doubly transitive, a minimal transition block $(w,n,M)$ of $\pi_1$ occurs in $y'_{[k,k+|w|)}=\vec{y}_{[k,k+|w|)}=w$ for some $k\ge j$. Let the point $x'$ be routable through a symbol $a\in M$ at time $k+n$. There is a point $\vec{x}$ in $\pi_1^{-1}(\vec{y})$ which is also routable through $a$ at time $k+n$. Reset $\vec{x}_{[k,\infty)}$ to have $\vec{x}_{k+n}=a$ and $\vec{x}_t=x'_t$ for all $t\ge k+|w|$.

    Since $y$ is doubly transitive, the minimal transition block $(w,n,M)$ also occurs in $y_{[-l,-l+w|)}=\vec{y}_{[-l,-l+|w|)}=w$ for some $l\ge|w|$. Let the point $\vec x$ be routable through a symbol $b\in M$ at time $-l+n$. By Lemma \ref{lem:unique_routability} and Theorem \ref{thm:mutually_separated} there is exactly one $x^{(i)}$ among $x^{(1)},\cdots,x^{(c_1)}$ which is routable through $b$ at time $-l+n$. Reset $\vec{x}_{(-\infty,-l+|w|)}$ to have $\vec{x}_{-l+n}=b$ and $\vec{x}_t=x^{(i)}_t$ for all $t\le -l$. Therefore we have $\vec{x}$ left asymptotic to $x^{(i)}$, right asymptotic to $x'$, and $\pi(\vec{x})=\pi_2(\vec{y})=z$. Corollary \ref{cor:transition_equals_equivalence} implies that $x'\sim x^{(i)}$ with respect to $\pi$. It follows that $\pi_1^{-1}(C)$ contains at most $c_1$ doubly transitive $\pi$-preimages of $z$ which are not equivalent to each other with respect to $\pi$.

    Now let $C_1,C_2,\cdots,C_{c_2}$ be all the transition classes over $z$ with respect to $\pi_2$. Note that by Corollary \ref{cor:contain_dbly_transition_pt} the class degree of $\pi$ is the maximal number of doubly transitive points in $\pi^{-1}(z)=\bigcup_{j=1}^{c_2}\pi_1^{-1}(C_j)$ which are not equivalent to each other with respect to $\pi$. By the above argument, each $\pi_1^{-1}(C_j)$ contains at most $c_1$ doubly transitive points which are not equivalent to each other with respect to $\pi$. Therefore we have $c_\pi\le c_1c_2$.
\end{proof}

We finish with the following question, which can be regarded as a measure-theoretical version of Theorem \ref{thm:contain_transition_pt}.

\begin{que}
    Let $(X,Y,\pi)$ be an irreducible factor triple and let $\nu$ be an ergodic measure on $Y$. Given a right transitive point $y\in Y$ which is $\nu$-generic, does each transition class over $y$ contain a generic point of a measure of relative maximal entropy over $\nu$?
\end{que}

Note that the class $C \in \C(y)$ may not contain generic points for \emph{different} measures of relative maximal entropy over $\nu$. For example consider the factor code $\pi_1$ on the full 2-shift in Example \ref{ex:multiplication_fails}. Then $c_{\pi_1} = 2$ and $\pi_1$ maps the $(1/3,2/3)$ and $(2/3,1/3)$-Bernoulli measures to the same measure on $X$. However each transition class over a point in $X$ is a singleton.

\vspace{0.1cm}
\begin{ack*}
The first author was supported by Fondecyt project 3120137, the second author was supported by Fondecyt project 3130718, and the third author was supported by Basic Science Research Program through the National Research Foundation of Korea(NRF) funded by the Ministry of Education (2012R1A1A2006874).
The authors would like to thank Michael Schraudner and the referee for helpful comments.
\end{ack*}

\bibliographystyle{abbrv}
	\bibliography{BiblioMahsa}

\begin{thebibliography}{10}

\bibitem{AllQ13}
M.~Allahbakhshi and A.~Quas.
\newblock Class degree and relative maximal entropy.
\newblock {\em Trans. Amer. Math. Soc.}, 365(3):1347--1368, 2013.

\bibitem{Ash90}
J.~Ashley.
\newblock Bounded-to-{$1$} factors of an aperiodic shift of finite type are
  {$1$}-to-{$1$} almost everywhere factors also.
\newblock {\em Ergodic Theory Dynam. Systems}, 10(4):615--625, 1990.

\bibitem{Bla57}
D.~Blackwell.
\newblock The entropy of functions of finite-state {M}arkov chains.
\newblock In {\em Transactions of the first {P}rague conference on information
  theory, {S}tatistical decision functions, random processes held at {L}iblice
  near {P}rague from {N}ovember 28 to 30, 1956}, pages 13--20. Publishing House
  of the Czechoslovak Academy of Sciences, Prague, 1957.

\bibitem{Boy83}
M.~Boyle.
\newblock Lower entropy factors of sofic systems.
\newblock {\em Ergodic Theory Dynam. Systems}, 3(4):541--557, 1983.

\bibitem{Boy86}
M.~Boyle.
\newblock Constraints on the degree of a sofic homomorphism and the induced
  multiplication of measures on unstable sets.
\newblock {\em Israel J. Math.}, 53(1):52--68, 1986.

\bibitem{BoyP11}
M.~Boyle and K.~Petersen.
\newblock Hidden {M}arkov processes in the context of symbolic dynamics.
\newblock In {\em Entropy of hidden {M}arkov processes and connections to
  dynamical systems}, volume 385 of {\em London Math. Soc. Lecture Note Ser.},
  pages 5--71. Cambridge Univ. Press, Cambridge, 2011.

\bibitem{BoyT84}
M.~Boyle and S.~Tuncel.
\newblock Infinite-to-one codes and {M}arkov measures.
\newblock {\em Trans. Amer. Math. Soc.}, 285(2):657--684, 1984.

\bibitem{BurR58}
C.~J. Burke and M.~Rosenblatt.
\newblock A {M}arkovian function of a {M}arkov chain.
\newblock {\em Ann. Math. Statist.}, 29:1112--1122, 1958.

\bibitem{ChaU03}
J.-R. Chazottes and E.~Ugalde.
\newblock Projection of {M}arkov measures may be {G}ibbsian.
\newblock {\em J. Statist. Phys.}, 111(5-6):1245--1272, 2003.

\bibitem{CovP74}
E.~M. Coven and M.~E. Paul.
\newblock Endomorphisms of irreducible subshifts of finite type.
\newblock {\em Math. Systems Theory}, 8(2):167--175, 1974/75.

\bibitem{GatP97}
D.~Gatzouras and Y.~Peres.
\newblock Invariant measures of full dimension for some expanding maps.
\newblock {\em Ergodic Theory Dynam. Systems}, 17(1):147--167, 1997.

\bibitem{Hed69}
G.~A. Hedlund.
\newblock Endomorphisms and automorphisms of the shift dynamical system.
\newblock {\em Math. Systems Theory}, 3:320--375, 1969.

\bibitem{Jun13}
U.~Jung.
\newblock Fiber-mixing codes and factors of {G}ibbs measures for shifts of
  finite type.
\newblock {\em preprint}.

\bibitem{KitMT91}
B.~Kitchens, B.~Marcus, and P.~Trow.
\newblock Eventual factor maps and compositions of closing maps.
\newblock {\em Ergodic Theory Dynam. Systems}, 11(1):85--113, 1991.

\bibitem{LM}
D.~Lind and B.~Marcus.
\newblock {\em An introduction to symbolic dynamics and coding}.
\newblock Cambridge University Press, Cambridge, 1995.

\bibitem{PetQS03}
K.~Petersen, A.~Quas, and S.~Shin.
\newblock Measures of maximal relative entropy.
\newblock {\em Ergodic Theory Dynam. Systems}, 23(1):207--223, 2003.

\bibitem{Shi11}
S.~Shin.
\newblock Pseudo-cyclic renewal systems.
\newblock {\em Theoret. Comput. Sci.}, 412(39):5387--5399, 2011.

\bibitem{Tho04}
K.~Thomsen.
\newblock On the structure of a sofic shift space.
\newblock {\em Trans. Amer. Math. Soc.}, 356(9):3557--3619, 2004.

\bibitem{Tro90}
P.~Trow.
\newblock Degrees of finite-to-one factor maps.
\newblock {\em Israel J. Math.}, 71(2):229--238, 1990.

\bibitem{Wal86}
P.~Walters.
\newblock Relative pressure, relative equilibrium states, compensation
  functions and many-to-one codes between subshifts.
\newblock {\em Trans. Amer. Math. Soc.}, 296(1):1--31, 1986.

\bibitem{Yoo10}
J.~Yoo.
\newblock On factor maps that send {M}arkov measures to {G}ibbs measures.
\newblock {\em J. Stat. Phys.}, 141(6):1055--1070, 2010.

\bibitem{Yoo11}
J.~Yoo.
\newblock Measures of maximal relative entropy with full support.
\newblock {\em Ergodic Theory Dynam. Systems}, 31(6):1889--1899, 2011.

\end{thebibliography}

\end{document}